\documentclass[a4paper,egregdoesnotlikesansseriftitles,final]{scrartcl}
\usepackage[utf8]{inputenc}
\usepackage[british]{babel}
\usepackage[T1]{fontenc}
\usepackage{lmodern}
\usepackage{microtype}
\usepackage{amsfonts,amsthm,amssymb,amsmath}
\usepackage{mathtools}
\usepackage{letterswitharrows}
\usepackage{xcolor}
\usepackage{graphicx}
\usepackage{tikz}
\usetikzlibrary{arrows}
\usepackage[colorlinks,unicode,bookmarks,pdfusetitle,linkcolor={red!60!black},citecolor={green!60!black},urlcolor={blue!60!black},bookmarksnumbered,bookmarksopen]{hyperref}
\usepackage[capitalize,nameinlink]{cleveref}
\crefname{enumi}{}{}
\crefname{property}{property}{properties}
\crefname{LEM}{Lemma}{the Lemmas}
\usepackage{enumitem}
\usepackage{thmtools, thm-restate} 
\usepackage{cite}
\usepackage[abbrev]{amsrefs}
\usepackage{doi}
\renewcommand{\PrintDOI}[1]{\doi{#1}}

\newtheorem{THM}{Theorem}[section]
\newtheorem{LEM}[THM]{Lemma}
\newtheorem{COR}[THM]{Corollary}
\newtheorem{QUESTION}[THM]{Question}

\theoremstyle{definition}

\newcommand{\abs}[1]{\lvert#1\rvert}

\newcommand{\menge}[1]{\left\{#1\right\}}

\newcommand{\tn}[1]{\textnormal{#1}}
\renewcommand{\phi}{\varphi}

\newcommand{\R}{\mathbb{R}}

\newcommand{\join}{\lor}
\newcommand{\meet}{\land}
\newcommand{\sub}{\subseteq}
\newcommand{\sm}{\smallsetminus}

\newcommand{\cB}{\mathcal{B}}

\newcommand{\cJ}{\mathcal{J}}

\newcommand{\cO}{\mathcal{O}}

\DeclareMathOperator{\DM}{\textbf{DM}}

\newcommand{\ucl}[1]{\lfloor #1 \rfloor}
\newcommand{\dcl}[1]{\lceil #1 \rceil}

\title{The Structure of Submodular Separation~Systems}
\author{Christian Elbracht \and Jakob Kneip \and Maximilian Teegen}
\date{24th March 2021\\ (updated 28th Sept.\ 2021)}
\begin{document}
\maketitle
\begin{abstract}
We analyse various structural and order-theoretical aspects of abstract separation systems and partial lattices, as well as the relationship between the different submodularity conditions one can impose on them. 
\end{abstract}

\section{Introduction}
As part of their graph minor project \cite{GMX}, Robertson and Seymour introduced tangles as a tool  to capture indirectly `highly connected regions' of a graph. Subsequently this tool has been generalised to many other contexts \cites{ProfilesNew,AbstractTangles,TangleTreeAbstract,TangleTreeGraphsMatroids,ProfileDuality,FiniteSplinters,TanglesInMatroids,AbstractSepSys}. Part of this process of generalisation was to abstract from concrete separations of a graph to a set of more general objects, called an abstract separation system, of which the separations of a graph form but an instance. These abstract separation systems carry just enough structure to make the fundamental theorems of tangle theory work \cite{ProfilesNew,TangleTreeAbstract,ToTviaTTD}, and  they enable us to give unifying proofs of these main theorems of tangle theory for a variety of different settings. Moreover, since the definitions for abstract separation systems are just the bare minimum needed to make those  proofs work, these abstract separation systems become an interesting mathematical structure in their own right: any results about them immediately apply to a variety of contexts, including separations of graphs and matroids.

Formally, abstract separation systems are just a poset $\vS$ together with an order-reversing involution $ {}^\ast $. (The elements of $\vS$ are called \emph{separations} and the image of a separation under ${}^\ast$ is called its \emph{inverse}.)
A central property that is required of separation systems in almost every context is some form of \emph{submodularity}: a property needed to make the separation system `rich enough' to prove the desired theorems, for example the tree-of-tangles theorem \cite{ProfilesNew}.
Our aim in this paper is to study, and relate, various forms of submodularity -- always with an eye on its uses in the proofs of the central tangle theorems.

Originally, in \cite{ProfilesNew,TangleTreeAbstract,TanglesInMatroids,GMX}, it was not only required that $\vS$ was part of a \emph{universe of separations} -- a separation system $\vU \supseteq \vS$ which forms a lattice -- but also that there exist a submodular order function $f\colon \vU \to \R^+_0$ and $k \in \R^+_0$ such that \[
    \vS = \vS_k \coloneqq \{ \vs \in \vU \mid f(\vs) < k \}.
\]
Here, $f$ being \emph{submodular} means that $f(\vr) + f(\vs) \geq f(\vr \join \vs) + f(\vr \meet \vs)$ for all $\vr,\vs\in\vU$, and $f$ being an \emph{order function} means that it is invariant under taking inverses.

Diestel, Erde and Weißauer~\cite{AbstractTangles} showed that the theorems of tangle theory could also be deduced without relying on such an order function, demanding instead just one structural property of $\vS\subseteq\vU$ which in the case of the sets $\vS_k$ is imposed by the submodularity of $f$: 
that for all $\vr,\vs \in \vS$ at least one of $\vr\join\vs$ and $\vr\meet\vs$ is also in $\vS$.
Note that this structural property of $\vS$ is measured externally:
in the universe $\vU$, where the join $\vr\join\vs$ and the meet $\vr\meet\vs$ are taken.
To reflect this, we say that $\vS \subseteq \vU$ is \emph{submodular in $\vU$}.
Whenever a submodular order function $f$ on $\vU$ and a number $k$ exist such that $\vS = \vS_k$ for this order function, we say that the submodularity of $\vS$ in $\vU$ is \emph{order-induced} in $\vU$.

The authors of this paper have, in much of their own work, relied heavily on such structural submodularity of separation systems, rather than on the existence of a submodular order function. Indeed, separation systems which are submodular in some universe of separations form the most relevant class of separation systems nowadays, and the most general theorems of abstract tangle theory are formulated in their context \cites{AbstractTangles,FiniteSplinters,CanonicalToT,ToTviaTTD}.

The most natural structural notion of submodularity, however, is simply to call
a separation system $\vS$ \emph{submodular} if any two separations $\vr, \vs \in \vS$ have either a supremum $\vr \join \vs$ or an infimum $\vr \meet \vs$ in $\vS$.
Unlike in our earlier definition of submodularity for $\vS$ in some universe~$\vU$, the question now is whether such infima and suprema exist -- not whether they lie in~$\vS$.

Note that every separation system $\vS$ that is submodular in some universe $\vU$ of separations is also submodular in this sense, since every infimum or supremum of $\vr,\vs \in \vS$ in~$\vU$ which, by submodularity in $\vU$, also lies in $\vS$, is also the supremum or infimum of $\vr$ and $\vs$ in~$\vS$.
Submodularity of a separation system $\vS$, as defined locally in $\vS$ itself, is therefore a weakening of submodularity in some surrounding universe of separations.

One can then ask whether this weaker kind of submodularity still suffices as a basis for the theorems of tangle theory, which traditionally assume that the separation system $\vS$ whose tangles are studied is submodular in some universe~$\vU$.
Our first result, which we prove in \cref{sec:DM}, shows that it does:

\begin{restatable}{MAINTHM}{corsubmodular}
    \label{thm:submodularcor}
    Every submodular separation system is submodular in some universe of separations.
\end{restatable}

\Cref{thm:submodularcor} allows us to apply the main
theorems of tangle theory to separation systems which are known to be submodular only in the weaker local sense, without the need to re-prove them under this weaker assumption.

In \cref{sec:counterex_orderfun,sec:extend} we turn our attention to the question of when the submodularity of a separation system in a universe $\vU$ is always induced  by a submodular order function on $\vU$.
In \cref{sec:counterex_orderfun} we prove that it need not be:
\begin{restatable}{MAINTHM}{counterexorderfun}\label{thm:example_bip}
  There exists a separation system $\vS$ which is submodular in a universe $\vU$ of set bipartitions whose submodularity in $\vU$ is not induced by a submodular order function on $\vU$.
\end{restatable}
More precisely, we present a necessary condition for the submodularity of a separation system in a universe~$\vU$  to be order-induced in $\vU$, and use this to give concrete examples of systems which are submodular in some universe $\vU$ of separations but whose submodularity is not order-induced in this $\vU$.

In \cref{sec:extend} we consider another aspect of order-induced submodularity.
Whether the submodularity in a universe $\vU$ of a separation system is order-induced or not depends, a priori, on the choice of $\vU$.
As a simple example, consider the case that a separation system $\vS$ is submodular in a universe $\vU$ of separations, and that $\vU$ is a subuniverse of some larger universe $\vU'$ of separations. Then $\vS$ is submodular also in $\vU'$.
If the submodularity of $\vS$ in $\vU$ is witnessed by some submodular order function on $\vU$, we may ask whether we can extend this function to $\vU'$ to witness that $\vS$ is submodular also in~$\vU'$.
We show that this can be done in some cases.
The general question of whether it is always possible to extend such a witnessing submodular order function to a larger universe remains open.

Finally, in \cref{sec:subdec}, we present two decomposition theorems for separation systems that are submodular in distributive universes.
Our first decomposition theorem allows us to write every such separation system $\vS$ as a (not necessarily disjoint) union of three smaller ones, each of which is not only again submodular in the same universe, but is also closed under taking existing corners in $\vS$.
Thus, we cover $\vS$ by smaller, simpler, `spanned' subsystems.
To prove this, we introduce a variation of Birkhoff's representation theorem for universes of separations instead of lattices.
Moreover, in our decomposition theorem, the subsystems can be chosen disjoint, unless the separation system to be decomposed is one of set bipartitions.

Separation systems that are submodular in the (natural) universe $\vU$ of bipartitions of a set $V$ cannot be decomposed disjointly into submodular subsystems.
Indeed every non-empty subsystem would have to contain the separations $(V, \emptyset)$ and $(\emptyset, V)$, since these form opposite corners of every pair of inverse separations.
By submodularity in $\vU$ one of these -- and hence also the other as its inverse -- would have to lie in this subsystem.

Separation systems of set bipartitions are, however, very concrete and better understood than the more general abstract separation systems.
We may view these bipartition systems as the `elementary parts' which make up the separation systems that are submodular in distributive universes.
Applying our decomposition theorem repeatedly, for as long as disjoint decompositions are possible, we can thus break down every separation system that is submodular in a distributive universe into those elementary subsystems.
\begin{restatable}{MAINTHM}{decomposeinbipartitions}
    Every separation system $\vS$ which is submodular in some distributive universe~$\vU$ of separations is a disjoint union of corner-closed subsystems $\vS_1, \dots, \vS_n$ of $\vS$ (which are thus also submodular in $\vU$)
  each of which can be corner-faithfully embedded into a universe of bipartitions.
  \ifnum\value{section}>1

  Specifically, these subsystems are the equivalence classes of the relation $\sim$ on $\vS$ where $\vs \sim \vt$ if and only if $\vs \meet \sv = \vt \meet \tv$ in~$\vU$.
  \fi
\end{restatable}

Careful analysis of the proof of our decomposition theorem allows us to explicitly specify the subsystems.

The research in this paper was inspired, in part, by our search for a solution to the unravelling problem, which is concerned with another property in which the different kinds of submodularity differ. We encourage the reader to take a look at this paper's sibling,~\cite{Unravel}, to learn more about this problem.

\section{Preliminaries}
\label{sec:prelim}

In this paper we will use terminology from lattice theory as well as the theory of abstract separation systems.
We also introduce some definitions specific to this paper, most of which are generalisations of definitions made for separation systems and universes of separations to posets and lattices, respectively.
The essential terminology, which we will use throughout, is presented in this section.

All structures (posets, lattices, universes of separations etc.) in this paper are assumed to be finite unless explicitly stated otherwise.
\subsection{Lattice theory}
Let us begin with some terminology from lattice theory. We largely follow the notation of \cite{LatticeBook}.

A \emph{lattice} is a non-empty partially ordered set (or `poset') $L$ in which any two elements $a,b\in L$ have a supremum and an infimum, that is, there is a unique element $a\join b$ (their \emph{join} or \emph{supremum}) minimal such that $a\le a\join b$ and $b\le a\join b$ and a unique element $a\meet b$ (their \emph{meet} or \emph{infimum}) maximal such that $a\ge a\meet b$ and $b\ge a\meet b$.

Each lattice has a unique \emph{top} and a unique \emph{bottom} \emph{element}, that is an element $\top\in L$ with $a\le \top$ for every $a\in L$ and an element $\bot\in L$ with $\bot\le a$ for every $a\in L$. 

Two lattices are isomorphic, if they are isomorphic as partial orders. In particular, joins and meets are preserved under isomorphisms of lattices.
A \emph{sublattice} $L'$ of $L$ is a subset of $L$ which is closed under pairwise joins and meets in $L$.


A lattice is \emph{distributive} if it satisfies the distributive laws, that is for all $a,b,c\in L$ we have that $a\join (b\meet c)=(a\join b)\meet (a\join c)$ and $a\meet (b\join c)=(a\meet b)\join (a\meet c)$.

A typical example of a distributive lattice consists of the subsets of some set $V$ ordered by $\subseteq$.
Here the join of two sets is their union, the meet is their intersection.
We call this lattice the \emph{subset lattice of $V$}.

In fact, all finite distributive lattices can be represented as a set of subsets where $\join$ and $\meet$ coincide with union and intersection.
This is a fundamental result of lattice theory known as the \emph{Birkhoff representation theorem}, which we can state after the following additional definitions: a non-bottom element $x \in L$ is \emph{join-irreducible} if whenever $x = a \join b$ for some $a,b \in L$, then $x \in \{a,b\}$. The set of all join-irreducible elements of $L$ is denoted $\mathcal{J}(L)$ and forms a partially ordered set with the order inherited from $L$.
Given a partially ordered set $(P, \leq)$, the down-closed sets in $P$ form a distributive lattice with $\subseteq$ as the partial order, union as join and intersection as meet. This lattice is denoted as $\mathcal{O}(P)$.

\begin{restatable}[Birkhoff representation theorem; cf.\ \cite{LatticeBook}*{\S 5.12}]{THM}{birkhoff}
    \label{thm:birkhoff_book}
    Let $L$ be a finite distributive lattice. The map $\eta\colon L \to \cO(\cJ(L))$ defined by $
        \eta(a) = \menge{ x \in \cJ(L) \mid x \leq a } = \dcl{a}_{\cJ(L)}
    $ is an isomorphism of lattices.
\end{restatable}

Given a lattice $L$, any subset $P\subseteq L$ together with the restrictions of $\join$ and $\meet$ (as \emph{partial functions}) is called a \emph{partial lattice}. \cite{GraetzerLatticeThry}

\subsection{Separation systems}
The foundations of abstract separation systems are summarised in \cite{AbstractSepSys}.
What follows are definitions which we adopt from there.

A \emph{separation system} is a partially ordered set $\vS$ with an order-reversing involution~$^\ast$.\footnote{In the context of order theory separation systems are known under several different names, such as \emph{involution posets} \cite{BLYTH19721}.}
For an element $\vs \in \vS $, a \emph{separation}, we write $\sv$ as shorthand for $\vs^*$ and we denote the set $\{\vs,\sv\}$ of a separation together with its inverse just as the \emph{unoriented separation} $s$. The set of unoriented separations of $\vS$ is denoted as $S$, and vice versa.
When there is no risk of confusion, we use terms defined for unoriented separations also for their orientations, and the other way around.
Notably, we also speak of `the separation system $S$.'


If a separation system $\vU$ is a lattice, we say that $U$ is a \emph{universe (of separations)}.
In this case DeMorgan's laws holds: \[
    (\vr \join \vs)^* = \rv \meet \sv \;\text{ and }\; (\vr \meet \vs)^* = \rv \join \sv \; \text{ for all $\vr,\vs\in\vU$.}
\]

Given two separations $\vr, \vs \in \vU$, the separations $\vr \join \vs$, $\rv \join \vs$, $\vr \join \sv$ and $\rv \join \sv$ are called the \emph{corners of $\vr$ and $\vs$ in $\vU$}.

An \emph{isomorphism} between two separation systems $\vS$ and $\vS'$ is a bijection that preserves $\le$ and~$^\ast$.
Note that, in the case of a universe, such an isomorphism automatically preserves joins and meets.

Note that, given a set $V$, the subset lattice on $V$ is an example of a universe of separations via the involution $A^\ast=V\sm A$ for $A\subseteq V$. We call this universe the \emph{bipartition universe $\cB(V)$} on $V$. If we consider $\cB(V)$, we also write $(A,V\sm A)$ instead of just $A$ for oriented separations and we denote the unoriented separation corresponding to $(A,V\sm A)$ and $(V\sm A,A)$ as $\{A,V\sm A\}$.

Another example of a universe of separations defined on a set $V$ is the set of \emph{set separations of $V$}, which is given by all the sets $\{A,B\}$ of subsets $A,B$ of $V$ such that $A\cup B=V$. The orientations of such a separation $\{A,B\}$ are then the oriented pairs $(A,B)$ and $(B,A)$, with involution $(A,B)^\ast=(B,A)$. We let $(A,B)\le (C,D)$ if and only if $A\subseteq C$ and $B\supseteq D$. The operations $\join$ and $\meet$ are given by $(A,B)\join (C,D)=(A\cup C,B\cap D)$, $(A,B)\meet (C,D)=(A\cap C,B\cup D)$.

\subsection{Submodularity and additional terminology}
Finally, we set up some terminology specific to this paper, mostly generalised versions of established terminology.

A separation system $\vS$ is a \emph{subsystem} of another separation system $\vS'$ if $\vS \subseteq \vS'$ and the involution on $\vS$ is the restriction of the involution on $\vS'$.
In particular, $\vS$ is a subset of $\vS'$ which is closed under the involution on $\vS'$.
If a subsystem $\vS$ of a universe $\vU$ is closed under joins and meets in $\vU$, we say that $\vS$ (together with the restrictions of $\join$ and $\meet$) is a \emph{subuniverse} of $\vU$.
For example, the bipartition universe $\cB(V)$ on a set $V$ is a subuniverse of the universe of set separations of $V$.


\cite{AbstractSepSys} considers submodular order functions for universes of separations, we will need the more general notion of such a function for arbitrary lattices.
Given a lattice $L$, a function $f\colon L\to \R^+_0$ 
is called \emph{submodular} if \[f(a\join b)+f(a\meet b)\le f(a)+f(b)\] for all $a,b\in L$.

If $U$ is a universe, then a submodular function on $\vU$ as a lattice is called a \emph{submodular order function} if it is \emph{symmetric}, that is $f(\vu)=f(\uv)$ for all $\vu\in \vU$. (Such a function can also be interpreted as a function from $U$ to $\R^+_0$ and we shall thus write $f(u)$ to mean $f(\vu)$.) This agrees with the notation from \cite{AbstractSepSys}.

We say that a subset $ P $ is \emph{order-induced submodular} in a lattice $ L $, or that \emph{the submodularity of $P$ in $L$ is order-induced in $L$}, if there exists some submodular function $f$ on $L$,
    such that $P = \{ a \in L \mid f(a) < k \}$ for some $k$.
In this case, we also say that $f$ \emph{induces the submodularity of $P$ in $L$} and that $f$ and $k$ \emph{induce the submodularity of $P$ in~$L$}.

Similarly, we say that a subsystem $\vS$ of a universe $\vU$ is \emph{order-induced submodular} in $\vU$, if there exists a submodular order function $f$ on $\vU$ and some $k\in \R^+_0$ which induce the submodularity of $\vS$ in $\vU$, i.e., such that $\vS =\vS_k$ where $\vS_k:= \{ \vs \in U \mid f(\vs) < k \}$. 


Given a poset $P'$ and some subset $P$ of $P'$ we say that $P$ is \emph{submodular in} $P'$, if, for all $a,b\in P$ the following holds:
there is a supremum of $a$ in $b$ in $P$ and this supremum also lies in $P'$, or\footnote{Note that this `or' is not exclusive.}
there is an infimum of $a$ in $b$ in $P$ and this infimum also lies in $P'$.\footnote{Recall that a \emph{supremum} of $a$ and $b$ inside a poset is a (unique) least common upper bound of $a$ and $b$,
and an \emph{infimum} a (unique) greatest common lower bound.}

If $P$ is submodular in $P$, we say that $P$ is \emph{submodular}. Note that, if $P$ is submodular in a lattice $L$, then for any $a,b\in P$ we have either $a\join b\in P$ or $a\meet b\in P$.

Analogously, given a subsytem $\vS$ of separation system $\vS'$, we call $\vS$ \emph{submodular in} $\vS'$ if the poset $\vS$ is submodular in the poset $\vS'$.
For a universe of separations $\vU$, a \emph{submodular subsystem $\vS\subseteq \vU$} shall be any subsystem $\vS$ of $\vU$ which is submodular in $\vU$.
Note that the notion of $\vS$ being \emph{(structurally) submodular in $\vU$} from \cite{AbstractSepSys} is equivalent to our notation of $\vS$ being a submodular subsystem of $\vU$.

\section{Witnessing submodularity externally} 
\label{sec:DM}
The traditional, external, notions of submodularity always require our separation system $\vS$ to be part of a universe  $ \vU $ of separations, even though, often, we are interested in~$ \vS $ and substructures therein only and do not particularly care about the shape of~$ \vU $. The only reason for keeping this ambient universe~$ \vU $ around is that we need to be able to express joins and meets of elements of~$ \vS $, and decide whether these lie in- or outside of~$ \vS $. The mathematical arguments exploiting the submodularity of~$ \vS $ never truly make use of~$ \vU $, but only of the knowledge that at least one of two opposing corner separations, $\vr \join \vs$ and $\rv \join \sv = (\vr\meet\sv)^*$, is always present in~$ \vS $, for all $\vr,\vs\in\vS$.

As discussed in the introduction, the simplest, most general, and thus most natural, form of submodularity for a separation system is intrinsic from its poset structure, where
 a poset~$ P=(P,\le) $ is submodular if all pairs~$ a,b\in P $ have a supremum or an infimum in~$ P $.
Yet almost all theorems in the theory of abstract separation systems are phrased in terms of some form of submodularity which is external, in some universe of separations, even when that universe bears no particular relevance on the result.

In this section we offer a way out: a method by which the submodularity of some $ \vS $ in itself can be reflected into a suitable universe of separations in such a way, that its submodularity is expressed externally.
If~$ \vS $ is a separation system which is submodular in some universe~$ \vU $, then~$ \vS $ is also submodular on its own. Here, we will show a converse to this: if a separation system~$ \vS $ is submodular on its own, then we can construct a universe~$ \vU $ which contains an isomorphic copy of the separation system~$ \vS $, i.e., there is an \emph{embedding} of $\vS$ into $\vU$.
We can choose this embedding in such a way, that the pre-existing joins and meets inside~$ \vS $ are preserved. More precisely, if~$ \vr $ and~$ \vs $ have a supremum~$ \vt $ in~$ \vS $, then after embedding~$ \vS $ into~$ \vU $ we will have~$ \vt=\vr\join\vs $, where the latter is measured in~$ \vU $. Thus our $\vU$ is chosen so that $\vS$ is \emph{submodular in} $\vU$.


\begin{THM}
    \label{thm:submodular}
	For every separation system~$ \vS $, finite or infinite, there exists a universe~$ \vU $ of separations and an embedding~$ \phi\colon\vS\to\vU $, with the property that~$ \phi(\vt)=\phi(\vr)\join \phi(\vs) $ if and only if~$ \vt $ is the supremum of~$ \vr $ and~$ \vs $ in~$ \vS $, and likewise~$ \phi(\vu)=\phi(\vr)\meet \phi(\vs) $ if and only if~$ \vu $ is the infimum of~$ \vr $ and~$ \vs $ in~$ \vS $.
     Moreover, if $\vS$ is finite, then $\vU$ can be chosen to be finite.
	
    In particular, if~$ \vS $ is submodular, then~$ \phi(\vS) $ is submodular in $\vU$.
\end{THM}

The heavy lifting of~\cref{thm:submodular}'s proof is done by employing the \emph{Dedekind-MacNeille-completion}~\cite{MacNeille}, a lattice theoretic tool with which one can embed an arbitrary poset into a suitable lattice while preserving any pre-existing finite joins and meets.\footnote{The Dedekind-MacNeille-completion is more commonly used for infinite lattices, where it is used to embed a lattice into a complete lattice, hence the name. It is also known as the \emph{completion by cuts}.}
Our task then is to equip the resulting completion of the poset~$ \vS $ with an involution which turns it into a universe of separations, and which makes the embedding of~$ \vS $ into its Dedekind-MacNeille-completion an isomorphism onto its image.

To define this Dedekind-MacNeille-completion, we follow the notation of~\cite{LatticeBook}. Let~$ P $ be any poset, finite or infinite. Given a subset~$ X\sub P $ we write~$ X^\ell $ for the set of lower bounds of~$ X $ in~$ P $: the set of all~$ p\in P $ such that~$ p\le x $ for all~$ x\in X $. Similarly we write~$ X^u $ for the set of all upper bounds of~$ X $ in~$ P $. To improve readability we will omit braces when concatenating these operations, e.g., we shall write~$ X^{u\ell} $ rather than~$ (X^u)^\ell $.

The~\emph{Dedekind-MacNeille-completion} of~$ P $ is then given by
\[ \DM(P)\coloneqq\menge{X\sub P\mid X^{u\ell}=X} \]
using~$ \sub $ as the partial order. A result by MacNeille~\cite{MacNeille} asserts that~$ \DM(P) $ is indeed a lattice and, moreover, the map~$ \phi\colon P\to\DM(P) $ given by
\[ \phi(p)\coloneqq \menge{p}^\ell \]
is an embedding of the poset~$ P $ into~$ \DM(P) $ with the property that~$ \phi(r) $ is the supremum (resp.\ infimum) of~$ \phi(p) $ and~$ \phi(q) $ if and only if~$ r $ is the supremum (resp.\ infimum) of~$ p $ and~$ q $ in~$ P $. (Compare \cite{LatticeBook}*{Theorem~7.40}.)

To build some intuition about the Dedekind-MacNeille-completion, observe that for a singleton~$ \menge{p} $, the set~$ \menge{p}^u $ is simply the up-closure~$ \ucl{p} $ of~$ p $ in~$ P $. Moreover an element~$ q $ of~$ P $ is a lower bound of the up-closure of some~$ p $ precisely if~$ q\le p $, and hence~$ \menge{p}^{u\ell}=\menge{p}^\ell=\dcl{p} $. In particular, when applying any series of~$ {}^u $ and~$ {}^\ell $ to a singleton set~$ \menge{p} $, only the very last operation is relevant: for instance~$ \menge{p}^{\ell u\ell}=\menge{p}^{\ell} $, which shows that the map~$ \phi $ indeed takes its image in~$ \DM(P) $.

Let us now prove~\cref{thm:submodular}.

\begin{proof}[Proof of~\cref{thm:submodular}.]
	Let~$ \vS=(\vS,\le,{}^*) $ be a separation system. Let~$ \vU=\DM(\vS) $ be the Dedekind-MacNeille-completion of the poset~$ \vS $ with the embedding~$ \phi\colon\vS\to\vU $ given by~$ \phi(\vs)=\menge{\vs}^\ell $.
	
	For a set~$ X\sub\vS $ we write $ X^* $ for the point-wise involution~$ \menge{\xv\mid\vx\in X} $ of~$ X $. For readability we shall extend our convention to omit braces to include~$ {}^* $,~$ {}^u $, and~$ {}^\ell $. Clearly~$ X^{**}=X $ for all~$ X\sub\vS $.
	
	We define an involution~$ {}^\circledast $ on~$ \vU $ by letting
	\[ X^\circledast \coloneqq X^{u*} \]
	and claim that this turns~$ \vU $ into a universe of separations and~$ \phi $ into an isomorphism of separation systems between~$ \vS $ and its image in~$ \vU $. To verify this claim we need to ascertain the following: that~$ {}^\circledast $ takes its image in~$ \vU=\DM(\vS) $; that~$ {}^\circledast $ is an involution; that~$ {}^\circledast $ is order-reversing; and finally that~$ \phi $ commutes with the involution, i.e.\ that~$ \phi(\vs)^\circledast=\phi(\sv) $.
	
	Before we do this, observe that since the involution~$ {}^* $ of~$ \vS $ is order-reversing we have
	\[ X^{u*}=X^{*\ell}\qquad\tn{ and }\qquad X^{\ell*}=X^{*u} \]
	for all~$ X\sub\vS $. We shall be using these two equalities throughout the remainder of the proof.
	
	To see that~$ {}^\circledast $ takes its image in~$ \vU $, note that for~$ X\in\vU $ we have
	\[ (X^\circledast)^{u\ell}=X^{u*u\ell}=X^{u\ell*\ell}=X^{*\ell}=X^{u*}=X^\circledast\,, \]
	where the third equality used the definition of~$ \vU=\DM(\vS) $ to infer~$ X^{u\ell}=X $. Thus we indeed have~$ X^\circledast\in\vU $ by definition of~$ \vU=\DM(\vS) $.
	
	The map~$ {}^\circledast $ is an involution since
	\[ (X^\circledast)^\circledast=X^{u*u*}=X^{u\ell**}=X^{u\ell}=X\,, \]
	for~$ X\in\vU $, using again the definition of~$ \vU=\DM(\vS) $.
	
	To see that~$ {}^\circledast $ is order-reversing let~$ X,Y\in\vU $ with~$ X\sub Y $ be given; we need to show that~$ X^\circledast \supseteq Y^\circledast $. From~$ X\sub Y $ it follows that~$ X^u\supseteq Y^u $, which in turn implies~$ X^{u*}\supseteq Y^{u*} $. Thus indeed~$ X^\circledast\supseteq Y^\circledast $.
	
    We now show that~$ \phi(\vs)^\circledast=\phi(\sv) $ for all~$ \vs\in\vS $. So let~$ \vs\in\vS $ be given. Recall that~$ \menge{\vs}^{\ell u}=\menge{\vs}^u $ and $\phi(\vs) = \menge{\vs}^\ell$. Using this equality we find that
	\[ \phi(\vs)^\circledast=\phi(\vs)^{u*}=\menge{\vs}^{\ell u*}=\menge{\vs}^{u*}=\menge{\vs}^{*\ell}=\menge{\sv}^\ell=\phi(\sv)\,, \]
	as claimed.

    Since $\phi$ preserves the existing pairwise suprema and infima of the poset $\vS$, it is thus the desired embedding.
\end{proof}

We can phrase \cref{thm:submodular} more concisely, as follows:

\corsubmodular*

\section{Structural submodularity which is not order-induced} 
\label{sec:counterex_orderfun}
In this section we deal with the question of whether the submodularity of a submodular subsystem $\vS \subseteq \vU$ of a universe $\vU$ is always induced by some submodular order function $f$ on $\vU$, i.e., that $\vS = \vS_k$ for some $k$.
We will answer this question in the negative, even for distributive $\vU$, and thus show that submodularity in a universe is a proper generalisation of order-induced submodularity.

We consider the question first for partial lattices $P\subseteq L$ which are submodular in some lattice $L$. Recall that these are partial lattices $P \subseteq L$ such that for any two points $a,b \in P$ at least one of $a\join b$ and $a\meet b$ (taken in $L$) is in $P$.

One way to show that the submodularity of a given partial lattice is not order-induced is to find a sequence $a_1,a_2,\dots a_n$ of elements of a lattice $L$ so that every submodular function $f$ on $L$ for which $P$ is an $S_k$ would need to satisfy $f(a_1)<f(a_2)<\dots f(a_n)<f(a_1)$. Such a sequence may be found by finding a directed cycle in a digraph $D$ on $L$ where we draw an edge from $a$ to $b$ whenever every suitable submodular function on $L$ needs to satisfy $f(a)>f(b)$.

This motivates the following definition:
for $P \subseteq L$ we define the \emph{dependency digraph} $D = (L, E)$ of $P$ as a directed graph where $(a,b)$ is an edge in $E$ if and only if one of the following holds:
    \begin{itemize}
        \item $a\in L\sm P$ and $b\in P$;
        \item $a,b\in P$ and there is some $c\in P$ such that either
         \begin{itemize}
        \item $b = a \join c$ and $a \meet c \notin P$, or
        \item $b = a \meet c$ and $a \join c \notin P$;
    \end{itemize}
        \item $a,b\notin P$ and there is some $c\in P$ such that either
         \begin{itemize}
        \item $b = a \join c$ and $a \meet c \notin P$, or
        \item $b = a \meet c$ and $a \join c \notin P$.
    \end{itemize}
    \end{itemize}

Let us first show that given an order-induced submodular partial lattice $P\subseteq L$, the edges in the dependency digraph indeed witness that their start vertex has higher order than their end vertex.

\begin{LEM}
    If $P \subseteq L$ is order-induced submodular, witnessed by some $f$ and $k$,
    and $(a, b)$ is an edge in the dependency digraph of $P$,
    then $f(a) > f(b)$.
\end{LEM}
\begin{proof}
    Let $(a, b)$ be an edge in the dependency digraph. If $a\in L\sm P$ and $b\in P$ then $f(a)>f(b)$ since $f$ induces the submodularity of $P$ in $L$.
    
    If $a,b\in P$ we may assume without loss of generality that the edge between $a$ and $b$ exists because of some $c \in P$ with
    $b = a \join c$ and $a \meet c \notin P$.

    Because $f$ induces the submodularity of $P$ in $L$ we have $f(a \meet c) > f(c)$. Since $f$ is submodular \[
        f(a \join c) + f(a \meet c) \leq f(a) + f(c),
    \]
    and hence $f(b) = f(a \join c) < f(a)$, as required.
    
    Similarly, if $a,b\notin P$ we may assume without loss of generality that the edge between $a$ and $b$ exists because of some $c \in P$ with
    $b = a \join c$ and $a \meet c \notin P$.

    Because $f$ induces the submodularity of $P$ in $L$ we have $f(a \meet c) > f(c)$. Again, since $f$ is submodular \[
        f(a \join c) + f(a \meet c) \leq f(a) + f(c),
    \]
    and hence $f(b) = f(a \join c) < f(a)$, as required.
\end{proof}
Thus a directed cycle in the dependency digraph is an obstruction to the order-induced submodularity of $P$.
\begin{COR}
  If the dependency digraph of $P$ contains a directed cycle then $P$ is not order-induced submodular.
\end{COR}
Since every cycle in the dependency digraph $D$ of $P$ is completely contained in either $D[P]$ or $D[L\sm P]$, we sometimes consider these two subgraphs independently from each other, naming them  the \emph{inner dependency digraph $D[P]$} and the \emph{outer dependency digraph $D[L\sm P]$}

Each cycle in the dependency digraph has length at least $3$:
\begin{LEM}
  Let $P \subseteq L$ be submodular in $L$, then the dependency digraph of $P$ contains no 
  directed cycle of length two.
\end{LEM}
\begin{proof}
As stated above, a cycle of length 2 cannot contain one vertex in $P$ and one in $L\sm P$.
Thus if the dependency digraph $D$ contains a cycle of length 2 between $a$ and $b$, then by the definition of the dependency digraph $a$ and $b$ are comparable in $\le$, so $a\le b$, say.
Note that either $a,b\in P$ or $a,b\notin P$.
In either case, as $(a,b)$ is an edge in $D$, there exists a $c\in P$ such that $a\join c=b$ and $a\meet c\not\in P$. Similarly, there exists a $d\in P$ such that $b\meet d=a$ and $b\join d\notin P$.

If $c\le d$ then $d\ge a$ and $d\ge c$ and thus $a\join c=b\le d$ contradicting the assumption that $b\join d\notin P$. Similarly, if $d\le c$ then $d\le c\le b$, again contradicting the assumption. Hence $c$ and $d$ are incomparable and thus $c\join d\in P$ or $c\meet d\in P$, as $c,d\in P$ and $P$ is submodular in $L$.
However, $b=a\join c\le d\join c$, thus $d\join c\ge b$, hence $d\join c\ge b\join d$, but also $d\join c\le d\join b$ as $c\le b$, and thus $d\join c=d\join b\notin P$. And similarly,  $a=d\meet b\ge d\meet c$, thus $d\meet c\le a\meet c$ but also $d\meet c\ge a\meet c$ and thus $d\meet c=a\meet c\notin P$.
 
 Thus $D$ cannot contain a cycle of length 2.
\end{proof}

Using the dependency digraph, we can give an example of a lattice $L$ together with a partial lattice $P\subseteq L$ which is submodular in $L$, but where this submodularity is not order-induced. Our example will use a universe of separations as its lattice, and a submodular separation system for the partial lattice.

In fact, our example consists of oriented bipartitions (equivalently: subsets) on a set of six elements.
The Hasse diagram of this example is displayed in \cref{fig:example_bip};
a formal description follows.

\begin{figure}[!hp]
  \begin{center}
      \includegraphics[width=.95\textwidth,trim=0 30 0 30,clip]{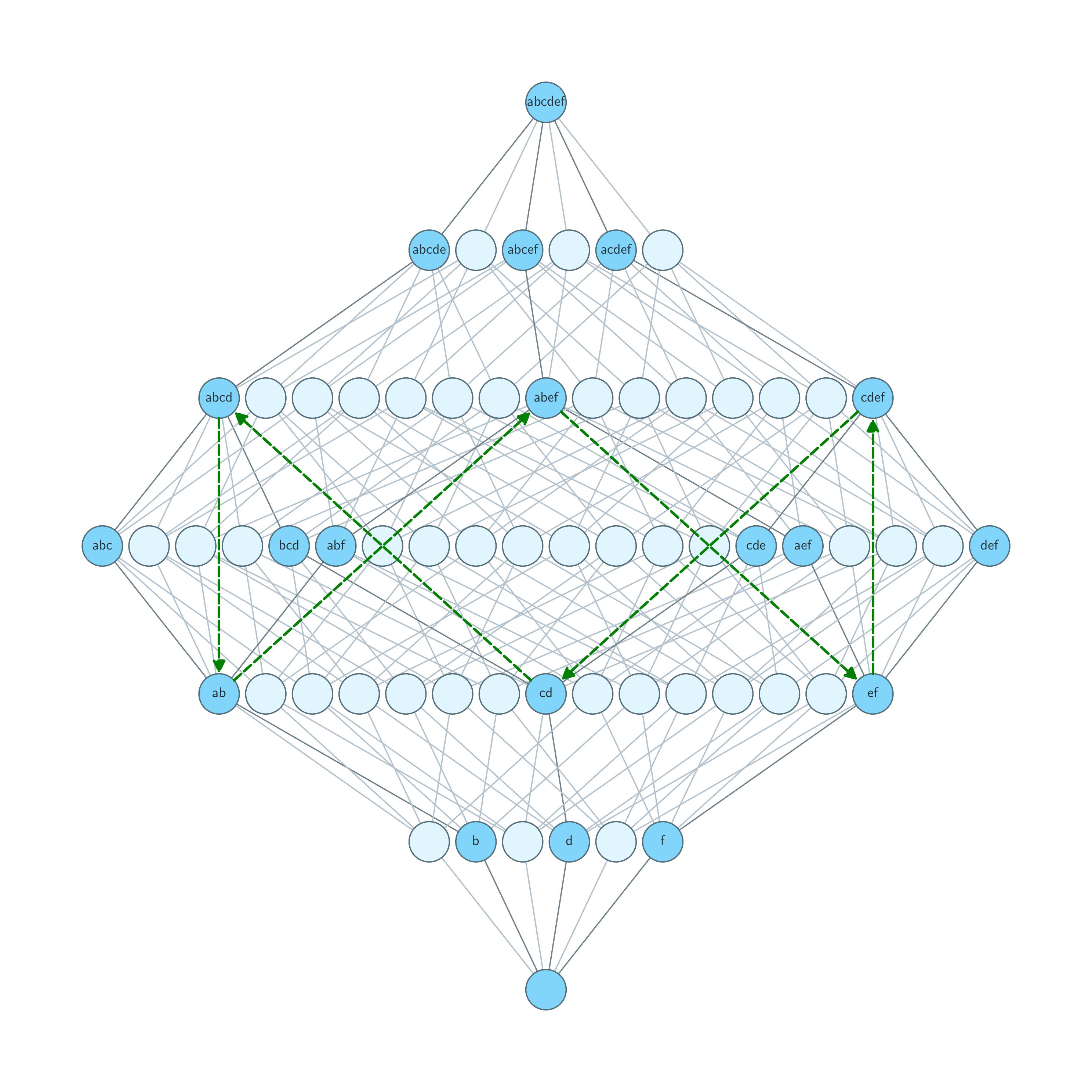}
  \caption{The Hasse diagram of $\vU$ from \cref{thm:example_bip}. For readability, only points in $\vS$ are labelled and only one side of each bipartition is denoted.}\label{fig:example_bip}
  \end{center}
  \end{figure}

Consider the universe $\vU = \cB(V)$ of bipartitions of $V = \{a,b,c,d,e,f\}$. In there we consider the separation system $\vS$ consisting of the orientations of the following unoriented bipartitions:
  \begin{align*}S=\{&\{\emptyset, V\},\\&\{\{b\},\{a,c,d,e,f\}\},\{\{d\},\{a,b,c,e,f\}\},\{\{f\},\{a,b,c,d,e\}\}\},\\&\{\{a,b\},\{c,d,e,f\}\},\{\{c,d\},\{a,b,e,f\}\},\{\{e,f\},\{a,b,c,d\}\},\\&\{\{a,b,c\},\{d,e,f\}\},\{\{a,b,f\},\{c,d,e\}\},\{\{a,e,f\},\{b,c,d\}\}\}.\end{align*}
  It is easy to see that $\vS$ is submodular in $\vU$. However, the dependency digraph of $\vS$ in $\vU$ contains the directed cycle
  \begin{multline*}(\{a,b,c,d\},\{e,f\})\rightarrow(\{a,b\},\{c,d,e,f\})\rightarrow(\{a,b,e,f\},\{c,d\})\rightarrow(\{e,f\},\{a,b,c,d\})\\\rightarrow(\{c,d,e,f\},\{a,b\})\rightarrow(\{c,d\},\{a,b,e,f\})\rightarrow(\{a,b,c,d\},\{e,f\}).\end{multline*}
  For example, there is an arc between $(\{a,b,c,d\},\{e,f\})$ and $(\{a,b\},\{c,d,e,f\})$ since \[(\{a,b,c,d\},\{e,f\})\meet (\{a,b,f\},\{c,d,e\})=(\{a,b\},\{c,d,e,f\})\] and \[(\{a,b,c,d\},\{e,f\})\join (\{a,b,f\},\{c,d,e\})=(\{a,b,c,d,f\},\{e\}),\] but $(\{a,b,c,d,f\},\{e\})$ is not an element of $\vS$.
  The existence of the remaining arcs in the cycle can be checked similarly.

This example proves the following theorem:

\counterexorderfun*

One might wonder if every example of a partial lattice with a cycle in its dependency digraph actually contains a cycle in the \emph{inner} dependency digraph. This is not the case, as an example we show the Hasse digram of such a lattice in \cref{fig:example_outer} and indicate the partial lattice inside this lattice as well as the cycle in the dependency digraph.
\begin{figure}[h]
\begin{center}
      \includegraphics[scale=.5,trim=0 20 0 20,clip]{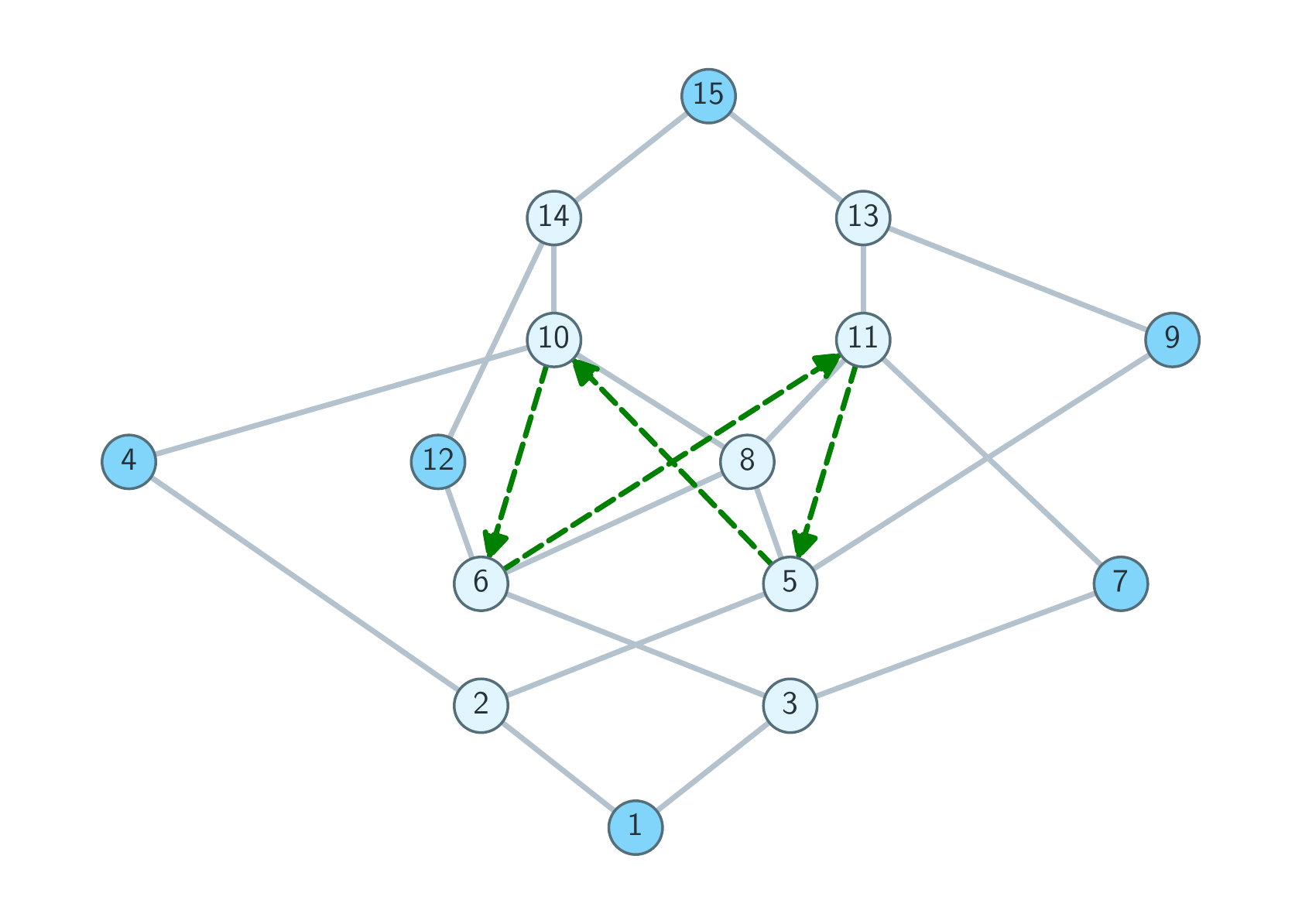}
\end{center}
\caption{The dark blue elements form a partial lattice, which does not contain a cycle in the inner dependency digraph, however the green dashed edges form a cycle in the outer dependency digraph}\label{fig:example_outer}
\end{figure}

However, we are not aware of any examples of submodular separation systems whose submodularity in a universe is not order-induced and whose dependency digraph is acyclic:
\begin{QUESTION}\label{que:depdisepsys}
 Does there exists a separation system  $\vS\subseteq \vU$ which is submodular in $\vU$, such that the dependency digraph of $\vS$ does not contain a cycle, but the submodularity of $\vS$ in $\vU$ is not order-induced?
\end{QUESTION}
We can ask the same question for a submodular partial lattice:
\begin{QUESTION}\label{que:depdilattice}
 Does there exists a partial lattice  $P\subseteq L$ which is submodular in the lattice $L$ such that the dependency digraph of $P$ does not contain a cycle, but the submodularity of $P$ in $L$ is not order-induced?
 \end{QUESTION}
These two questions are, in fact, equivalent. To see this, observe that a positive answer to \cref{que:depdisepsys} implies a positive answer to \cref{que:depdilattice}: if there exists a separation system  $\vS\subseteq \vU$ which is submodular in $\vU$, such that the dependency digraph of $\vS$ does not contain a cycle, but the submodularity is not order-induced, then we can consider $\vS$ as a partial lattice inside the lattice $\vU$ which still does not contain a cycle in its dependency digraph. However, if $k\in \R_0^+$ and $f_l:\vU\to \R_0^+$ would be a submodular function witnessing that $\vS$ is order-induced submodular as a partial lattice, then we could consider the function $f$ given by $f(\vs)=f_l(\vs)+ f_l(\sv)$ for every $\vs\in \vS$, which would then be a submodular order function for $\vU$ as a universe, and $f$ and $2k$ induce the submodularity of $\vS$ in $\vU$.

On the other hand, if there exists a partial lattice  $P\subseteq L$ which is submodular in the lattice $L$ such that the dependency digraph of $P$ does not contain a cycle, but the submodularity is not order-induced, we can construct a universe $\vU$ and a submodular subsystem  $\vS\subseteq \vU$, so that the dependency digraph of $\vS$ does not contain a cycle, but the submodularity of $\vS$ in $\vU$ is not order-induced, as follows:
let $L'$ be a copy of $L$ with reversed partial order (i.e.\ the poset-dual of $L$).
We let $\vU$ be the disjoint union $L \sqcup L'$, where we additionally declare $\vr \le \vs$ for all $\vr \in L$ and $\vs \in L'$.
The involution on $\vU$ is defined by mapping an element of $L$ to its respective copy in $L'$ and vice versa.
It is easy to see that this is a universe of separations and that $\vS=P\cup P'$ (where $P\subseteq L$ is as above and $P'\subseteq L'$ is the image of $P$ in $L'$) is a submodular subsystem of $\vU$.\footnote{Note, that in $\vU$ every separation is either small or co-small, i.e., for every $\vs\in \vU$ either $\vs\le \sv$ or $\sv\le \vs$.}
Moreover, $\vS$ is not order-induced submodular, since we can restrict any witnessing submodular order function on $\vU$ to a submodular function on $L$, which would then witness that the submodularity of $P$ in $L$ is order-induced.

The dependency digraph of $\vS$ cannot contain a cycle either, since any such cycle would result in a cycle in the dependency digraph of $L$ or $L'$: every edge in the dependency digraph of $\vU$ either is also an edge in the dependency digraph of $L$ or $L'$, or is an edge between $L$ and $L'$ which needs to be an edge between an element of $\vU\sm \vS$ and $\vS$. Thus, given any cycle in the dependency digraph of $\vU$ which meets both $L$ and $L'$, we can consider a maximal subpath of this cycle contained in $L$; there then needs to be a directed edge in the dependency of $L$ between the last and the first vertex of this path.

\begin{figure}[h]
\begin{center}
      \includegraphics[scale=.5,trim=0 20 0 20,clip]{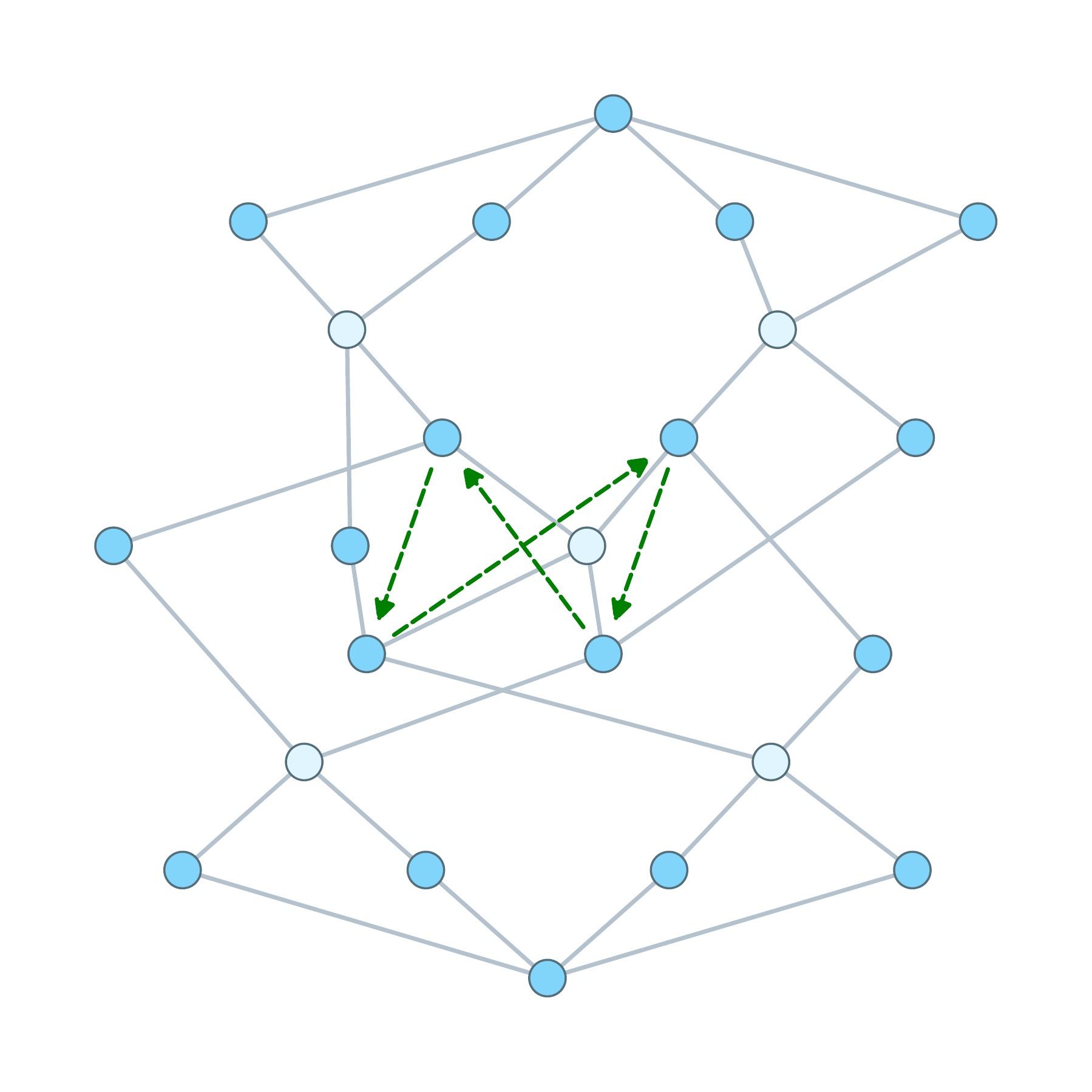}
\end{center}
\caption{The dark blue elements form a submodular partial lattice which, however it is embedded into a lattice, has the cycle indicated in green in its dependency digraph. }\label{fig:example_independent}
\end{figure}

Another question that one might ask is if at least in a weakened sense every submodular $\vS$ is an $\vS_k$:
whether for every submodular $\vS$ there exists a universe $\vU$ where $\vS$ can be embedded to be order-induced submodular.
After all, we know from \cref{sec:DM} that at the least there exists a universe in which $\vS$ is submodular, so it is reasonable
to suspect that we can even construct a universe where $\vS$ is order-induced submodular.

However, if we consider the inner dependency digraph of some $\vS$, then we can see that some of its edges will need to exist indepentently of the universe that $\vS$ is embedded into, as long as such an embedding preserves pre-existing (binary) joins and meets.
If some two elements of $\vS$ do not have, say, a supremum in $\vS$, then no matter what universe we embed $\vS$ into, their supremum in that universe will not lie in $\vS$, so the inner dependency digraph will have directed edges from either of these two elements to their infimum in $\vS$.

In \cref{fig:example_independent} we see the Hasse diagram of a submodular separation system, embedded into a universe, together with a cycle in its inner dependency digraph all of whose edges are of the type described above. This answers the question in the negative:

\begin{THM}
  There exists a submodular separation system which can not be embedded into any universe of separations in such a way that
  pre-existing binary joins and meets are preserved and such that it is order-induced submodular in that universe.
\end{THM}

\section{Extending a submodular function} \label{sec:extend}
Our aim in this section is to better understand for what kind of submodular separation systems the submodularity is order-induced.
We investigate inhowfar the existence of a submodular function depends on the surrounding universe $\vU$, that is, if we have an order function $f$ which induces the submodularity of some $\vS$ in a subuniverse $\vU' \subseteq \vU$, we ask whether we can extend $f$ to $\vU$ in such a way that it induces the submodularity of $\vS$ in $\vU$.

We give partial answers to this question: firstly that submodular functions can be extended in this way from an \emph{interval} in a universe and, secondly,
that for every subuniverse $\vU'$ of a universe $\vU$ there exists a submodular function $f$ and some $k$, such that $\vU' = f^{-1}([0,k])$.

It suffices to first consider these problems for submodular functions on lattices, rather than submodular order functions on universes of separations:
if $f'\colon \vU' \to \R^+_0$ is a submodular function on $\vU' \supseteq \vU$ which agrees on $\vU$ with some submodular \emph{order} function $f\colon \vU \to \R^+_0$,
then we can define a submodular \emph{order} function $\bar f$ on $\vU'$ which agrees with $f$ by setting \[
    \bar f(\vs) = \frac{f'(\vs) + f'(\sv)}{2}.
\]
We will then easily see that, in both cases, this function is as desired.

For the first theorem, recall that an \emph{interval} in a lattice $L$ is, for some $x,y\in L$, a subset $[x,y] = \{{s\in L} \mid x\le s\le y\}$. Every such interval forms a sublattice.
The following result shows that we can extend a submodular function defined on an interval.

\begin{THM}\label{lem:lattice_expand}
 Let $L$ be a lattice and $L'=[x,y]\subseteq L$ an interval in $L$.
 Suppose that $f\colon L'\to \R^+_0$ is a submodular function on $L'$ with maximum value $k$.
 Then there exists a submodular function $g\colon L\to \R^+_0$ such that $g(z)=f(z)$ for all $z\in L'$ and $g(z)>k$ for all $z\notin L'$.
\end{THM}
\begin{proof}
Let us denote as $L^{\downarrow}$ the set of all $z\in L\sm L'$ such that $z\le y$, as $L^{\uparrow}$ the set of all $z\in L\sm L'$ such that $z\ge x$ and as $L^{\leftrightarrow}$ the set of all $z\in L\sm L'$ such that neither $z\le y$ nor $z\ge x$.
Note that $L^{\downarrow}, L^\uparrow, L^\leftrightarrow$ and $L'$ together form a partition of $L$.

 For $z\in L$ such that $z\le y$ we define its \emph{down-level $dl(z)$} recursively as follows: assign $dl(\bot)=0$ for the bottom element $\bot$ of $L$. Now $dl(z):=\max\{dl(z')+1\mid z'<z\}$ for all other $z \in L$. Similarly, for $z\in LL$ such that $z\ge x$ we define its \emph{up-level $ul(z)$} recursively: we assign $ul(\top)=0$ for the top element $\top$ of $L$. Now $ul(z):=\max\{ul(z')+1\mid z'>z\}$.
 
 Let $l$ be the maximum possible level (up or down) and let $M=2^l\cdot k > k$.
 We now define $g$ as follows:
 \begin{align*}
  g(z)=
  \begin{cases}
  f(z) & z\in L'\\
  M\cdot (2-2^{-dl(z)})& z\in L^{\downarrow}\\
  M\cdot (2-2^{-ul(z)})& z\in L^{\uparrow}\\
  4\cdot M\, & z \in L^\leftrightarrow
  \end{cases}
 \end{align*}
To verify that this function is submodular we distinguish the possible cases which can occur for two incomparable elements $a,b \in L$. Note that  in the case of comparable elements, submodularity is trivially satisfied, so we suppose they are incomparable.
\begin{description}[style=nextline]
\item[The case $a,b\in L^{\leftrightarrow}$.]
 By construction, the maximal value of $g$ is $4\cdot M$, thus \[g(a\join b)+g(a\meet b)\le 4\cdot M+4\cdot M=g(a)+g(b).\]
\item[The case $a\in L^{\uparrow},b\in L^{\leftrightarrow}$.]
 By the definition of $L^{\uparrow}$, we have $a\join b\in L^{\uparrow}$ and $ul(a)>ul(a\join b)$, thus \[
 \begin{split}
 g(a\join b)+g(a\meet b) &\le M\cdot (2-2^{-ul(a\join b)})+4\cdot M \\ &< M\cdot (2-2^{-ul(a)})+4\cdot M =g(a)+g(b).
 \end{split}
 \]
\item[The case $a\in L^{\downarrow},b\in L^{\leftrightarrow}$.]
 Analogous to the above.
\item[The case $a\in L',b\in L^{\leftrightarrow}$.]
 By the definition of $L^{\uparrow}$, we have, since $a\join b\ge a\ge x$, that $a\join b\in L^{\uparrow}\cup L'$ and similarly, $a\meet b\in L^{\downarrow}\cup L'$. Thus, we have \[g(a\join b)+g(a\meet b)\le 2M+2M\le g(b)\le g(a)+g(b).\]
\item[The case $a,b\in L^{\uparrow}$.]
 Suppose without loss of generality that $ul(a)\le ul(b)$. By the definition of $L^{\uparrow}$ and $ul$, we have $a\join b\in L^{\uparrow}$ and $ul(a\join b)<ul(a)$.
 Furthermore $a\meet b\in L^{\uparrow}\cup L'$, so in any case $g(a\meet b)<2M$. We calculate 
  \begin{align*}g(a\join b)+g(a\meet b)&<M\cdot (2-2^{-(ul(a)-1)})+2M \\ &=4M-M(2^{-ul(a)}+2^{-ul(a)})\\&\le 4M-M(2^{-ul(a)}+2^{-ul(b)})=g(a)+g(b).\end{align*}
\item[The case $a,b\in L^{\downarrow}$.]
 Analogous to the above.
\item[The case $a\in L^{\downarrow},b\in L^{\uparrow}$.]
 By construction $a\meet b\in L^{\downarrow}$ and $a\join b\in L^{\uparrow}$. Moreover, by the definition of $g$ we have $g(a\meet b)\le g(a)$ and $g(a\join b)\le g(b)$ and thus \[g(a\meet b)+g(a\join b)\le g(a)+g(b).\]
\item[The case $a\in L',b\in L^{\uparrow}$.]
  By the definition of $L^{\uparrow}$, we have $a\join b\in L^\uparrow$. Moreover $ul(a\join b) < ul(b)$, by the definition of $g$ and choice of $M$, we thus have $g(a\join b)\le g(b)-k$. Additionally, $g(a\meet b)\in L'$, since $x\le a\meet b$ and $a\meet b\le a\le y$. Thus, by the definition of $k$, we have $g(a\meet b)\le g(a)+k$ and thus \[g(a\join b)+g(a\meet b)\le g(b)-k+g(a)+k=g(a)+g(b).\]
\item[The case $a\in L',b\in L^{\downarrow}$.]
   Analogous to the above.
\item[The case $a,b\in L'$.]
  Immediate, by the submodularity of $f$.
\end{description}
Since furthermore $g(z)>k$ whenever $z\in L\sm L'$, by the definition of $M$, the function $g$ is as claimed.
\end{proof}

This theorem will also serves as a tool in proving the second theorem, which is the following:
\begin{THM}\label{thm:sublattice_function}
 Let $L$ be a distributive lattice and $L'\subseteq L$ a sublattice. Then there exists a submodular function $f\colon L\to \R^+_0$ and a $k\in \R_0^+$ such that $L'=f^{-1}([0,k])$.
\end{THM}

\Cref{lem:lattice_expand} allows us to first prove \cref{thm:sublattice_function} only for the special case of sublattices $L'$ which include the top and bottom element of $L$, and to then handle general sublattices by combing that result with \cref{lem:lattice_expand}.

\begin{LEM}\label{thm:lattice_function_dense}
 Let $L$ be a distributive lattice and $L'\subseteq L$ a sublattice, such that $L$ and $L'$ have the same top and the same bottom element. Then there exists a submodular function $f\colon L\to \R^+_0$ such that $L'=f^{-1}(0)$.
\end{LEM}

\begin{proof}
 By the Birkhoff representation theorem (\cref{thm:birkhoff_book}) we may suppose without loss of generality that $L=\cO(P)$, for some poset $P$.
 We may thus interpret the elements of $L$ (and thus also those of $L'$) as subsets of $P$.

 For every element $p\in P$ let $E_p$ be the set of elements of $L'$ which contain $p$. In particular, the top element of $L$ lies in $E_p$, so $E_p$ is non-empty. Thus, we can consider, for every $p\in P$, the set $X_p$ given by $\bigcap_{X\in E_p} X$. Note that $p$ is an element of $X_p$. 

 Observe that, since $L'$ is a sublattice, we have $X_p\in L'$ for every $p$. Given some $Y\in L$ we define $f(Y)$ by summing, over all $p$ in $Y$, the number of elements of $X_p$ that do not lie in $Y$. Formally, \[
    f(Y) = \sum_{p \in Y} \abs{X_p \sm Y}.
 \]
 This function is submodular, since for all $X,Y \in L$ we can calculate as follows \begin{align*}
     &f(X) + f(Y) = \sum_{p \in Y} \abs{X_p \sm Y} + \sum_{p \in X} \abs{X_p \sm X} \\
                 &= \sum_{p \in X \cap Y} (\abs{X_p \sm Y} + \abs{X_p \sm X}) + \sum_{p \in Y \sm X} \abs{X_p \sm Y} + \sum_{p \in X \sm Y} \abs{X_p \sm X} \\
                 &= \sum_{p \in X \cap Y} (\abs{X_p \sm (X\cap Y)} + \abs{X_p \sm (X \cup Y)}) + \sum_{p \in Y \sm X} \abs{X_p \sm Y} + \sum_{p \in X \sm Y} \abs{X_p \sm X} \\
                 &= f(X \cap Y) + \sum_{p \in X \cap Y} \abs{X_p \sm (X \cup Y)}+ \sum_{p \in Y \sm X} \abs{X_p \sm Y} + \sum_{p \in X \sm Y} \abs{X_p \sm X} \\
                 &\ge f(X \cap Y) + \sum_{p \in X \cap Y} \abs{X_p \sm (X \cup Y)}+ \sum_{p \in Y \sm X} \abs{X_p \sm (X \cup Y)} + \sum_{p \in X \sm Y} \abs{X_p \sm (X \cup Y)} \\
                 &= f(X \cap Y) +  \sum_{p \in X \cup Y} \abs{X_p \sm (X \cup Y)}\\
                 &= f(X\cap Y) + f(X \cup Y).
\end{align*}

 Thus all that is left to show is that $f(Y)>0$ for every $Y\in L\sm L'$. To see this, we observe that, since the bottom element lies in $L'$, any such $Y$ needs to contain some element $p$. If $X_p\subseteq Y$ for every $p\in Y$, then this would imply that $Y=\bigcup X_p$, contradicting the assumption that $Y\notin L'$. Thus there is some $p\in Y$ such that $X_p\not \subseteq Y$. In particular there needs to be some $q\in X_p$ such that $q\notin Y$, which witnesses that $f(Y)>0$.
\end{proof}

Combining \cref{thm:lattice_function_dense} and \cref{lem:lattice_expand} results in a proof of \cref{thm:sublattice_function}:
\begin{proof}[Proof of \cref{thm:sublattice_function}]
 Let $\bot$ be the bottom element of $L'$ and let $\top$ be the top element of $L'$. By \cref{thm:lattice_function_dense} there is a submodular function $f$ on $L' = [\bot,\top] \subseteq L$ and a $k\in \R$ such that $f^{-1}([0,k))=L'$. Using this $f$ as input in \cref{lem:lattice_expand} results in the desired submodular function on $L$.
\end{proof}

From \cref{lem:lattice_expand} and \cref{thm:sublattice_function} we now immediately obtain the same results for subuniverses, in the way discussed above:
\begin{THM}\label{thm:subuniverse_function}
 Given a distributive universe $\vU$ of separations and a subuniverse $\vU'\subseteq \vU$, there is a submodular order function $f\colon\vU\to \R^+_0$ and a $k\in \R^+_0$ such that $\vU'=f^{-1}([0,k])$.
\end{THM}
\begin{proof}
    We apply \cref{thm:sublattice_function} to $\vU'$ as a sublattice of $\vU$ to obtain a submodular function $f'$ and $k' \in \mathbb{R}_0^+$ with $\vU' = f'^{-1}([0,k'])$.
    We now define a symmetric order function $f$ on $\vU$ with $f(\vs) \coloneqq f'(\vs) + f'(\sv)$. With $k \coloneqq 2k'$ we have $\vU' = f^{-1}([0,k])$, as desired.
\end{proof}

\begin{THM}\label{lem:uinterval}
 Let $\vU$ be a universe of separations and $\vU' = [\vx,\xv]\subseteq \vU$ a symmetric interval in $\vU$.
 Suppose that $f\colon \vU'\to \R^+_0$ is a submodular order function on $\vU'$ with maximum value $k$.
 Then there exists a submodular order function $g\colon \vU\to \R^+_0$ such that $g(z)=f(z)$ for all $z\in \vU'$ and $g(z)>k$ for all $z\notin \vU'$.
\end{THM}
\begin{proof}
    We apply \cref{lem:lattice_expand} to $\vU'$ as an interval in the lattice $\vU$, to obtain a submodular function $g'$ on $\vU$ which agrees with $f$ on $\vU'$. This function need not be symmetric, but we can define $g(\vz) \coloneqq \frac{g'(\vz) + g'(\zv)}{2}$. Since $f$ is symmetric and $g'$ agrees with $f$ on $\vU'$, also $g$ agrees with $f$ on $\vU$. Moreover $g$ is symmetric. Since $g'$ takes values larger than $k$ outside of $\vU'$, so does $g$.
\end{proof}

\section{Submodular decompositions in distributive universes} 
\label{sec:subdec}
In this concluding section we consider decompositions of separation systems which are submodular in some universe, asking how  such a separation system can be written as the union of proper subsystems which are still submodular.
On one hand, we show that each separation systems $\vS$ which is submodular in some distributive universe $\vU$ of separations can be decomposed (although not necessarily disjoint) into at most three strictly smaller, again submodular in $\vU$, separation systems.
On the other hand, we will be able to deduce that we can decompose every such separation system into disjoint submodular subsystems, each of which can be embedded into a universe of bipartitions, in which they are again submodular.

The former statement also allows us to lower bound the size of a largest proper submodular subsystem:
by the pigeon-hole principle, at least one of these subsystems will have a size of at least $\frac{\abs{S}}{3}$.
This observation vaguely links the question of submodular decompositions to the unravelling problem \cite{Unravel}:
suppose $\vS$ contains a separation $\vs$ such that $\vS' = \vS \sm \{\vs,\sv\}$ is still submodular -- this is the case if $\vS$ can be unravelled -- then we can decompose $\vS$ into the two submodular subsystems $\vS'$ and $\{\vs, \sv, {\vs\join\sv}, {\vs\meet\sv}\}$.

However, while this is a decomposition into fewer parts than the ones we will obtain from our theorems, our decompositions will have the advantage that their constituent subsystems are not merely submodular in $\vU$ but `spanned' in $\vS$:
Given a universe $\vU$ of separations and a subsystem $\vS\subseteq \vU$, we say that $\vS'\subseteq \vS$ is a \emph{corner-closed subsytem of $\vS$ (in $\vU$)} if, for all $\vs,\vr\in \vS'$ 
we have $\vs\join \vr\in \vS'$ whenever $\vs\join \vr\in \vS$. 
In particular, if $\vS$ is submodular in $\vU$, then any corner-closed subsystem $\vS'\subseteq \vS$ is submodular in $\vU$ as well.

We begin by considering the special case of systems of bipartitions. This will later become a subcase in the proof of our general decomposition theorem.
The idea applied in the general case will also be similar to the one in the bipartition case.
To be able to transfer these techniques we will apply the Birkhoff representation theorem to a universes of separations and investigate how the involution of the universes interacts with this representation.
We will state this in the form of an extended Birkhoff theorem for universes of separations.

\subsection{Decomposition in bipartition universes}
Given the universe $\vU$ of bipartitions of some set $V$ and a separation system $\vS\subseteq \vU$ which is submodular in $\vU$,
we consider, for some $v,w\in V$, the set \[ \{(A,B)\in \vS\mid \{v,w\}\subseteq A\text{ or } \{v,w\}\subseteq B\}.\]
This set forms a corner-closed subsystem of $\vS$ in $\vU$.
We can utilise this observation to find a decomposition of $\vS$ into three proper subsystems.
\begin{THM}\label{thm:decompose}
    Given a universe $\vU = \cB(V)$ of bipartitions and a separation system $\vS\subseteq \vU$, such that $|S|\ge 3$, there are corner-closed subsystems $\vS_1,\vS_2,\vS_3\subsetneq \vS$, such that $\vS_1\cup \vS_2\cup \vS_3=\vS$.
\end{THM}
\begin{proof}
As $|S|\ge 3$, there are two distinct separations $\{A,B\},\{C,D\}\in S$ such that $A,B,C,D\neq \emptyset$. Moreover, we may assume that, after possibly exchanging $C$ and $D$, we have neither $C \subseteq A$ nor $C \subseteq B$ and thus $A\cap C\neq \emptyset$ and $B\cap C\neq \emptyset$. Additionally, after possibly exchanging $A$ and $B$, we may assume $B\cap D\neq \emptyset$.

Now pick $x\in A\cap C, y\in B\cap C$ and $z\in B\cap D$. Let $\vS_1$ be the set of all separations in $\vS$ not separating $x$ from $y$, let $\vS_2$ be the set of all separations in $\vS$ not separating $x$ from $z$ and let $\vS_3$ consists of all separations not separating $y$ from $z$. By construction, theses sets form corner-closed subsystems: a corner of two separations not separating $x$ from $y$, say, does not separate these two points either.

Moreover, $(A,B)$ is in neither $\vS_1$ nor $\vS_2$ and $(C,D)$ neither in $\vS_2$ nor $\vS_3$, thus $\vS_i \subsetneq \vS$ for all $1\le i\le 3$.

Finally, observe that, given any $(E,F)\in \vS$, either $E$ or $F$ contains two of the points $x,y,z$, so $(E,F)\in \vS_1\cup \vS_2\cup \vS_3$. Thus $\vS_1\cup \vS_2\cup \vS_3=\vS$, as claimed.
\end{proof}

\subsection{Birkhoff's theorem for distributive universes and decompositions in distributive universes}
To lift \cref{thm:decompose} to general distributive universes of separations, we will represent separations as subsets of some ground set.
For this we will once more, as in \cref{sec:extend}, use the Birkhoff representation theorem for distributive lattices:
\birkhoff*

If, in this theorem, the provided distributive lattice $L$ is actually a universe of separations, we obtain an order-reversing involution on $\cO(\cJ(L))$ by concatenating $\eta$ with the involution on the universe.
For our version of the Birkhoff theorem in distributive universes, we examine how this involution behaves with respect to $\cJ(L)$.

\begin{THM}[Birkhoff representation of universes of separations] \label{thm:birkhoff_uni}\text{}\\
    For every involution poset\footnotemark{} $(P, \leq, {}')$, the lattice $\cO(P)$ becomes a distributive universe of separations $(\cO(P), {}^\ast)$ when equipped with the involution ${}^\ast\colon X \mapsto P \sm X'$, where $X' = \menge{x' \mid x\in X}$.

    Let $U$ be a finite distributive universe of separations and let $P = \cJ(\vU)$.
    Then there exists an order-reversing involution~${}^\prime$ on $P$, such that the map $\eta\colon \vU \to \cO(P)$ defined by $
        \eta(a) = \menge{ x \in P \mid x \leq a } = \dcl{a}_{P}
    $ is an isomorphism of universes of separations between $\vU$ and $(\cO(P), {}^*)$.
\end{THM}
\footnotetext{Recall that involution posets are the same as separation systems. However, to emphasise that the involution on $\cJ(\vU)$ is different from the involution on $\vU$, despite $\cJ(\vU)$ being a subset of $\vU$, we prefer the term `involution poset' in this context.}

\begin{proof}[Proof of~\cref{thm:birkhoff_uni}]
The first statement is immediate.
 For the second part let us assume we are given a distributive universe $\vU$ of separations and need to construct an involution on $P \coloneqq \cJ(\vU)$ so that $\vU$ is isomorphic to $\cO(P)$.

\Cref{thm:birkhoff_book} tells us that the two are isomorphic as lattices, so it remains to take care of the involution.
Concatenating the isomorphism of lattices $\eta\colon \vU \to \cO(J(\vU))$ with the involution on $\vU$ gives us an involution $ {}^\ast $ on $ \cO(P) $ which is order-reversing.
Take note that $ {}^\ast $ maps down-closet subsets of $P$ to down-closed subsets of $P$; it is \emph{not\/} defined on the elements of $P$.

That ${}^\ast$ is order-reversing means that $ X\subsetneq Y $ if, and only if, $ X^\ast\supsetneq Y^\ast $ for all down-closed subsets $X,Y$ of~$ P $.
Our aim is to define an order-reversing involution $ ' $ on $ P $ so that for all $ X\in\cO(P) $ we have~$ X^\ast=P\sm\menge{x'\mid x\in X} $.
We begin with the following claim, which is also a necessary condition for this aim to be achievable:
\[\label{lem:inverse_size}
    \textit{For all $X\in \cO(P)$ we have that $|X^\ast|=|P|-|X|$.} \tag{\textdagger}
\]

\noindent
We prove \cref{lem:inverse_size} by contradiction.
So assume that $ X $ is an inclusion-wise minimal down-closed subset of $ P $ for which \cref{lem:inverse_size} does not hold.
(It clearly holds for the empty set.)
Take a maximal element $ x $ of $ X $ and consider the down-closed set~$ X-x $.
By choice of $ X $, we have $ \abs{(X-x)^\ast}=\abs{P}-\abs{(X-x)} $.
From $ X^\ast\subsetneq(X-x)^\ast $ it thus follows that~$ \abs{X^\ast}\le\abs{P}-\abs{X} $.

To see that this holds with equality, first observe that since there is no down-closed set $ Y $ with $ (X-x)\subsetneq Y\subsetneq X $ and neither is there a down-closed set $ Y^\ast $ with $ (X-x)^\ast\supsetneq Y^\ast\supsetneq X^\ast $.
However, if $ (X-x)^\ast\sm X^\ast $ had more than one element, then adding a minimal one among them to $ X^\ast $ would give such a set~$ Y^\ast $.
Hence $ X^\ast $ must be exactly one element smaller than $ (X-x)^\ast $, giving equality and contradicting the choice of $X$.
This proves \cref{lem:inverse_size}.
\medbreak

\noindent
Let us now define the involution $ ' $ on~$ P $. 
The following up- and down-closures are all to be taken in~$ P $.
For each $x \in P$ we define $x'$ to be the unique element of $ (\dcl{x} - x)^\ast\sm\dcl{x}^\ast $; this is well-defined by \cref{lem:inverse_size}. We will need to show that $'$ is an involution, that $'$ is order-reversing and that $X^* = P \sm \menge{x' \mid x \in X}$ for every down-closed set $X$.

We have $\dcl{x}^\ast\subseteq P\sm \ucl{x'}$, and hence $(P\sm \ucl{x'})^\ast\subseteq \dcl{x}$.
If we had proper inclusion, i.e.\ $(P\sm \ucl{x'})^\ast\subsetneq \dcl{x}$, then the down-closedness of $(P\sm \ucl{x'})^\ast$ would imply that $(P\sm \ucl{x'})^\ast\subseteq \dcl{x}-x$ and thus $(\dcl{x}-x)^\ast\subseteq P\sm \ucl{x'}$, contradicting the choice of $x'$.
Thus the inclusion holds with equality, and we have $\dcl{x}^\ast=P\sm \ucl{x'}$.

We are now going to show, given some down-closed set $ X $ in which $ x $ is maximal, that $ (X-x)^\ast\sm X^\ast = \menge{x'} $.
Since $\dcl{x}\subseteq X$, we have that $X^\ast\subseteq \dcl{x}^\ast$ and thus $X^\ast$ cannot contain $x'$.
But $(X-x)^\ast$ does contain $x'$, as otherwise, by $\dcl{x}^\ast=P\sm\ucl{x'}$, we have that $(X-x)^\ast\subseteq \dcl{x}^\ast$ and thus $(X-x)\supseteq \dcl{x}$, which is absurd.

This observation allows us to infer that $ ' $ is indeed an involution on $ P $: by the fact that $ \dcl{x}^\ast=(\dcl{x} - x)^\ast - x' $ is down-closed, we know that $ x' $ is maximal in $ (\dcl{x} - x)^\ast $ and $ x'' $ is the unique element of $ ((\dcl{x} - x)^\ast - x')^\ast\sm(\dcl{x} - x)^{\ast\ast}=\dcl{x}\sm(\dcl{x}-x) $, so $x''$ is  $x$.

Let us show that we have $ X^\ast=P\sm\menge{x'\mid x\in X} $ for all $ X\in\cO(P) $.
We do so by induction on the size of $ X $; for the empty set the statement is immediate.
So suppose that the assertion holds for each proper down-closed subset of some non-empty $ X \in \cO(P) $ and let $ x $ be a maximal element of $ X $.
Then $ (X-x)^\ast=P\sm\menge{y'\mid y\in (X-x)} $.
By the earlier observation, the single element in $ (X-x)^\ast\sm X^\ast $ is precisely $ x' $, giving $ X^\ast=P\sm\menge{y'\mid y\in X} $ as claimed.

Finally, we shall check that $ ' $ is order-reversing. For this let some $ x\in P $ be given.
Since $ \dcl{x}^\ast $ is a down-closed set which does not contain $ x' $ we have~$ \dcl{x}^\ast\sub P\sm\ucl{x'} $.
By applying $ {}^\ast $ to both sides and using the above paragraph we get that $ \dcl{x}\supseteq P\sm\menge{y'\mid y\in P\sm\ucl{x'}} $.
The right-hand side simplifies to $ \menge{y'\mid y\in\ucl{x'}} $.
Since this set is down-closed and contains $ x''=x $, the inclusion is in fact an equality, i.e.\ $ \dcl{x}=\menge{y\mid y'\in\ucl{x'}} $.
From this it follows that $ y\le x $ if and only if~$ y'\ge x' $.
\end{proof}

We are now ready to prove the central decomposition theorem, that every sufficiently large separation system which is submodular inside a distributive host universe of separations, can be either decomposed into three disjoint submodular subsystems, or is isomorphic to a subsystem of a universe of bipartitions while preserving existing corners (i.e.\ joins and meets).
Such an isomorphism $\iota\colon \vS \to \vS'$ between two subsystems $\vS \subseteq \vU$ and $\vS' \subseteq \vU'$ of universes $\vU$ and $\vU'$, where $\iota(\vr) \join \iota(\vs) = \iota(\vr \join \vs)$  whenever $\vr \join \vs \in \vS$, and conversely $\iota(\vr) \meet \iota(\vs) = \iota(\vr \meet \vs)$  whenever $\vr \meet \vs \in \vS$, for all $\vr,\vs\in\vS$, is called a \emph{corner-faithful embedding}.

\begin{THM}\label{thm:decompose_toostrong}
    Let $\vU$ be a distributive universe of separations and let $\vS\subseteq \vU$, $\abs{S} \geq 3$, be a separation system which is submodular in~$\vU$.
    Then there are corner-closed subsystems $\vS_1, \vS_2, \vS_3 \subsetneq \vS$ which are submodular in~$\vU$ and such that $\vS_1 \cup \vS_2 \cup \vS_3 = \vS$.

    Moreover $\vS_1, \vS_2, \vS_3$ can be chosen disjointly unless $\vS$ can be corner-faithfully embedded into a universe of bipartitions.
\end{THM}
\begin{proof}
    The proof is by induction on $\abs{\vU}$.

    By applying \cref{thm:birkhoff_uni} we may assume, without loss of generality, that $\vU = (\cO(P), {}^*)$ for some involution poset $(P, \leq, {}')$.
    For every $p \in P$ consider the sets
    \begin{align*}
        \vS_p &\coloneqq \menge{X \in \vS \mid p \in X,\, p' \notin X}, \\
        \vS_{p'} &\coloneqq \menge{X \in \vS \mid p \notin X,\, p' \in X}, \\
        \vS_{p,p'} &\coloneqq \vS \sm (\vS_p \cup \vS_{p'}).
    \end{align*}
    Note that these are pairwise disjoint, closed under involution, corner-closed and $\vS = \vS_{p,p'} \cup \vS_p \cup \vS_{p'}$.
    If for any $p$ these three sets form a non-trivial decomposition, we are done.
    Otherwise either for every $p \in P$ we have $\vS = \vS_{p,p'}$ or for some $p$ we have $\vS = \vS_p$.

    If for some $p$ we have $\vS = \vS_p$, then we can consider $\vS$ as a subsystem of $\vU' \coloneqq \cO(P \sm \menge{p, p'})$ under the corner-faithful embedding $\iota\colon \vS_p \to \vU',\, X \mapsto X - p$.
    Since $\abs{\vU'} < \abs{\vU}$ we can then apply the induction hypothesis to get the desired decomposition.

    If $\vS = \vS_{p,p'}$ for every $p \in P$, then this means, that for every $p$ we have $p \in X \Leftrightarrow p' \in X$ for all $X \in \vS$. In particular, for every $X$, we have $X^* = X \sm A' = X \sm A$.
    This means that $\vS$ is a submodular subsystem of the bipartition universe $\mathcal{B}(P)$, and \cref{thm:decompose} gives the desired decomposition.
\end{proof}

Observe that in $(\cO(P), {}^\ast)$ we have $ X \meet X^* = \menge{ p \in P \mid p \in X, p' \notin X} $.
Hence, by recursively applying the decomposition into $\vS_p, \vS_{p'}$ and $\vS_{p,p'}$ as above we never separate any $X$ and $Y$ where $X \meet X^* = Y \meet Y^*$.

Conversely, given any $X \in \cO(P)$, the set of all $Y \in \vS$ with $Y \meet Y^* = X \meet X^*$ is a corner-closed subsystem of $\vS$. By the last argument of the proof above, these can be considered as subsystems of  bipartition universes. We thus obtain our second decomposition result, while also explicitly specifying the subsystems that make up our decomposition:

\decomposeinbipartitions*

\bibliography{JMC}

\vspace{1cm}
\noindent
\begin{minipage}{\linewidth}
 \raggedright\small
   \textbf{Christian Elbracht},
   \texttt{christian.elbracht@uni-hamburg.de}

   \textbf{Jakob Kneip},
   \texttt{jakob.kneip@uni-hamburg.de}
   
   \textbf{Maximilian Teegen},
   \texttt{maximilian.teegen@uni-hamburg.de}

   Universität Hamburg,
   Hamburg, Germany
\end{minipage}
\end{document}